\documentclass[runningheads]{llncs}
\usepackage{graphicx}
\usepackage{amsmath} 

\usepackage{amssymb} 
 \usepackage{enumerate}
 
 \usepackage{mathtools}

\begin{document}
\title{Affine-invariant midrange statistics\thanks{This work has received support from the European Research Council under the Advanced ERC Grant Agreement Switchlet n.670645. Cyrus Mostajeran is supported by the Cambridge Philosophical Society.}}
%
%
\author{Cyrus Mostajeran\and
Christian Grussler\and
Rodolphe Sepulchre}
\authorrunning{C. Mostajeran et al.}
%
\institute{Department of Engineering,
University of Cambridge, 
Cambridge CB2 1PZ,
UK
}
\maketitle              
\begin{abstract}
We formulate and discuss the affine-invariant matrix midrange problem on the cone of $n\times n$ positive definite Hermitian matrices $\mathbb{P}(n)$, which is based on the Thompson metric. A particular computationally efficient midpoint of this metric is investigated as a highly scalable candidate for an average of two positive definite matrices within this context, before studying the $N$-point problem in the vector and matrix settings.

\end{abstract}
\section{Introduction}

In this paper, we develop a framework for affine-invariant midrange statistics on the cone of positive definite Hermitian matrices $\mathbb{P}(n)$ of dimension $n$. In Subsection \ref{subsec 1}, we briefly note the basic elements of Finsler geometry relevant to the problem. In Subsection \ref{subsec 2}, we define the affine-invariant midrange problem for a collection of $N$ points. In Section \ref{sec: 2}, we study a particular midrange of two positive definite matrices arising as the midpoint of a suitable geodesic curve. In Section \ref{sec: 3}, we briefly review the scalar and vector affine-invariant midrange problem before returning to the $N$-point problem on $\mathbb{P}(n)$ in Section \ref{sec: 4}.

\subsection{Affine-invariant Finsler metrics on $\mathbb{P}(n)$}
\label{subsec 1}

Consider the family of affine-invariant metric distances $d_{\Phi}$ on $\mathbb{P}(n)$ defined as 
\begin{equation} \label{d Phi}
d_{\Phi}(A,B)=\|\log A^{-1/2}BA^{-1/2} \|_{\Phi},
\end{equation}
where $\|\cdot\|_{\Phi}$ is any unitarily invariant norm on the space of Hermitian matrices of dimension $n$ defined by
$
\|X\|_{\Phi}:=\Phi(\lambda_1(X),\ldots,\lambda_n(X)),
$
with $\lambda_{\min}(X)=\lambda_n(X)\leq \ldots \leq \lambda_1(X) = \lambda_{\max}(X)$ denoting the $n$ real eigenvalues of $X$ and
$\Phi$ a symmetric gauge norm on $\mathbb{R}^n$ (i.e. a norm that is invariant under permutations and sign changes of the coordinates) \cite{Bhatia2003}. For any such $\Phi$, $X$ is said to be a $d_{\Phi}$-midpoint of $A$ and $B$ if
\begin{equation}
d_{\Phi}(A,X)=d_{\Phi}(X,B)=\frac{1}{2}d_{\Phi}(A,B).
\end{equation}
The curve $\gamma_{\mathcal{G}}:[0,1]\rightarrow\mathbb{P}(n)$ defined by 
\begin{equation} \label{R geodesic}
\gamma_{\mathcal{G}}(t)=A^{1/2}\left(A^{-1/2}BA^{-1/2}\right)^t A^{1/2}
\end{equation}
is geometrically significant as a minimal geodesic for any of the affine-invariant metrics $d_{\Phi}$ \cite{Bhatia2003}. The midpoint of $\gamma_{\mathcal{G}}$ is the matrix geometric mean $A\#B$, which is a metric midpoint in the sense of $d_{\Phi}(A,A\#B)=d_{\Phi}(A\#B,B)=\frac{1}{2}d_{\Phi}(A,B)$ for any $\Phi$. For the choice of $\Phi(x_1,\ldots,x_n)=(\sum_i x_i^2)^{1/2}$, $d_2:=d_{\Phi}$ corresponds to the metric distance generated by the standard affine-invariant Riemannian structure given by $\langle X,Y \rangle_{\Sigma}=\operatorname{tr}(\Sigma^{-1}X\Sigma^{-1}Y)$ for $\Sigma\in\mathbb{P}(n)$ and Hermitian matrices $X,Y\in T_{\Sigma}\mathbb{P}(n)$. For the choice of $\Phi(x_1,\ldots,x_n)=\max_i|x_i|$ on the other hand, $d_{\infty}:=d_{\Phi}$ yields the distance function that coincides with the Thompson metric \cite{Thompson1963} on the cone $\mathbb{P}(n)$
\begin{equation} \label{d inf}
d_{\infty}(A,B)=\|\log A^{-1/2}BA^{-1/2}\|_{\infty}=\max\{\log\lambda_{\max}(BA^{-1}),\,\log\lambda_{\max}(AB^{-1})\}.
\end{equation}
While the minimal geodesic $\gamma_{\mathcal{G}}$ in (\ref{R geodesic}) and the midpoint is unique for the Riemannian distance function $d_2$, it is not unique with respect to the $d_{\infty}$ metric which generally admits infinitely many minimal geodesics and midpoints between a given pair of matrices $A,B\in\mathbb{P}(n)$ \cite{Nussbaum1994}, as expected from the analogous picture concerning the associated norms in $\mathbb{R}^n$. As we shall see in Section \ref{sec: 2}, some of these minimal geodesics are much more readily constructible than others from a computational standpoint and yield midpoints that satisfy many of the properties expected of an affine-invariant matrix mean. Specifically, we will use the midpoint of a particular  minimal geodesic of $d_{\infty}$ from a construction by Nussbaum as a scalable relaxation of the matrix geometric mean that is much cheaper to construct than $A\#B$.

\subsection{The affine-invariant midrange problem}
\label{subsec 2}

Given a collection of $N$ points $Y_i$ in $\mathbb{P}(n)$, the midrange problem can be formulated as the following optimization problem
\begin{align}  \label{midrange}
\min_{X \succeq 0} \; \max_i \; d_{\infty}(X,Y_i).
\end{align}
We call a solution $X^{\star}$ to the above problem a midrange of $\{Y_i\}$.

\begin{proposition}
Let $X^{\star}$ be a solution to  (\ref{midrange}) with optimal cost $t^{\star} = f(X^{\star}) = \max_i \; d_{\infty}(X^{\star},Y_i)$. We have $l \leq t^{\star} \leq u$, where the lower and upper bounds are given by
\begin{equation}
l = \frac{1}{2}\operatorname{diam}_{\infty}(\{Y_i\}):=\frac{1}{2}\max_{i,j}d_{\infty}(Y_i,Y_j), \quad u = \min_i \max _j d_{\infty}(Y_i,Y_j) \leq 2l.
\end{equation}
\end{proposition}

\begin{proof}
By the triangle inequality, we have for any $i,j = 1, \ldots, N$,
\begin{equation} \label{bounds}
d_{\infty}(Y_i,Y_j) \leq d_{\infty}(Y_i,X^{\star}) + d_{\infty}(X^{\star},Y_j) \leq t^{\star} + t^{\star} = 2t^{\star},
\end{equation}
since $t^{\star}=\max_i d_{\infty}(X^{\star},Y_i)$. Taking the maximum of the left-hand side of (\ref{bounds}) over $i,j$, we arrive at $l=\frac{1}{2}\max_{i,j}d_{\infty}(Y_i,Y_j)\leq t^{\star}$. For the upper bound, note that taking $X=Y_i$ for each $i$, we obtain a cost $f(Y_i)=\max_j d_{\infty}(Y_i,Y_j)$. Since the minimum of these $N$ cost evaluations will still yield an upper bound on the optimum cost $t^{\star}$, we have $t^{\star} \leq u =  \min_i \max _j d_{\infty}(Y_i,Y_j)$.
\qed
\end{proof}

\section{The $2$-point midrange problem in $\mathbb{P}(n)$}
\label{sec: 2}

\subsection{Scalable $d_{\infty}$ geodesic midpoints}

As $\gamma_{\mathcal{G}}$ is a minimal geodesic for the $d_{\infty}$ metric, $A\#B$ lies in the set of midpoints of this metric. In particular,
\begin{equation}
A\#B\in \operatorname{argmin}_{X\in\mathbb{P}(n)}\left(\max_{Y\in\{A,B\}}d_{\infty}(X,Y)\right).
\end{equation}
Recall that the metric distance $d_{\infty}$ coincides with the Thompson metric $d_T$ on the cone $\mathbb{P}(n)$. If $C$ is a closed, solid, pointed, convex cone in a vector space $V$, then $C$ induces a partial order on $V$ given by $x\leq y$ if and only if $y-x\in C$. The Thompson metric on $C$ is defined as
$d_T(x,y)=\log\max\{M(x/y;C),M(y/x;C)\}$,
where $M(y/x;C):=\inf\{\lambda\in\mathbb{R}:y\leq \lambda x\}$ for $x\in C\setminus\{0\}$ and $y\in V$. For $A,B\in\mathbb{P}(n)$, we have $M(A/B)=\lambda_{\max}(AB^{-1})$, so that 
\begin{equation}
d_T(A,B)=\log\max\{\lambda_{\max}(AB^{-1}),\lambda_{\max}(BA^{-1})\},
\end{equation}
which indeed coincides with $d_{\infty}$ in (\ref{d inf}).

It is known that the Thompson metric does not admit unique minimal geodesics. Indeed, a remarkable construction by Nussbaum describes a family of geodesics that generally consists of an infinite number of curves connecting a pair of points in a cone $C$. In particular, setting $\alpha:=1/M(x/y;C)$ and $\beta:=M(y/x;C)$, the curve $\phi:[0,1]\rightarrow C$ given by
\begin{equation} \label{Nussbaum}
\phi(t;x,y):=\begin{dcases}
\left(\frac{\beta^t-\alpha^t}{\beta-\alpha}\right)y+\left(\frac{\beta\alpha^t-\alpha\beta^t}{\beta-\alpha}\right)x \quad &\mathrm{if} \; \alpha\neq\beta, \\
\alpha^t x &\mathrm{if} \; \alpha=\beta,
\end{dcases}
\end{equation}
is always a minimal geodesic from $x$ to $y$ with respect to the Thompson metric. The curve $\phi$ defines a projective straight line in the cone \cite{Nussbaum1994}. If we take $C$ to be the cone of positive semidefinite matrices with interior $\operatorname{int} C= \mathbb{P}(n)$, then for a pair of points $A,B\in\mathbb{P}(n)$, we have $\beta=M(B/A;C)=\lambda_{\max}(BA^{-1})$ and $\alpha=1/M(A/B;C)=\lambda_{\min}(BA^{-1})$. Thus, the minimal geodesic described by (\ref{Nussbaum}) takes the form
\begin{equation} \label{Nussbaum matrix}
\phi(t):=\begin{dcases}
\left(\frac{\lambda_{\max}^t-\lambda_{\min}^t}{\lambda_{\max}-\lambda_{\min}}\right)B+\left(\frac{\lambda_{\max}\lambda_{\min}^t-\lambda_{\min}\lambda_{\max}^t}{\lambda_{\max}-\lambda_{\min}}\right)A &\mathrm{if} \; \lambda_{\min}\neq\lambda_{\max}, \\
\lambda_{\min}^t A &\mathrm{if} \; \lambda_{\min}=\lambda_{\max},
\end{dcases}
\end{equation}
where $\lambda_{\max}$ and $\lambda_{\min}$ denote the largest and smallest eigenvalues of $BA^{-1}$, respectively. Taking the midpoint $t=1/2$ of this geodesic, we arrive at a computationally convenient $d_{\infty}$-midpoint of $A$ and $B$, which we will denote by $A*B$.
\begin{proposition}
For $A,B\in\mathbb{P}(n)$, we have
\begin{equation} \label{star}
A*B=\frac{1}{\sqrt{\lambda_{\min}}+\sqrt{\lambda_{\max}}}\left(B+\sqrt{\lambda_{\min}\lambda_{\max}} A\right).
\end{equation}
\end{proposition}

The result follows from elementary algebraic simplification upon setting $A*B=\phi(1/2;A,B)$ in the case $\lambda_{\min}\neq \lambda_{\max}$. If $\lambda_{\min}=\lambda_{\max}$, then $\phi(1/2;A,B)=\sqrt{\lambda_{\min}}A$ also agrees with the formula in (\ref{star}). The geometry of the set of $d_{\infty}$-midpoints of a given pair of points in $\mathbb{P}(n)$ is studied in detail in \cite{Lim2013}, where it is shown that there is a unique $d_{\infty}$ minimal geodesic from $A$ to $B$ if and only if the spectrum of $BA^{-1}$ lies in a set $\{\lambda,\lambda^{-1}\}$ for some $\lambda>0$. Moreover, it is shown that otherwise there are infinitely many $d_{\infty}$ minimal geodesics from $A$ to $B$, and that the set of $d_{\infty}$-midpoints of $A$ and $B$ is compact and convex in both Riemannian and Euclidean senses \cite{Lim2013}.

We now consider the merits of $A * B$ as a mean of $A$ and $B$. The following are a number of properties that should be satisfied by a sensible notion of a mean $\mu:\mathbb{P}(n)\times\mathbb{P}(n)\rightarrow\mathbb{P}(n)$ of a pair of positive definite matrices \cite{Bhatia,Kubo1980}.
\begin{enumerate}[(i)]
\item Continuity: $\mu$ is a continuous map.
\item Symmetry: $\mu(A,B)=\mu(B,A)$ for all $A,B \in \mathbb{P}(n)$.
\item Affine-invariance: $\mu(XAX^*,XBX^*)=X\mu(A,B)X^*$, for all $X\in GL(n)$.
\item Order property: $A \preceq  B \implies A\preceq  \mu(A,B) \preceq  B$.
\item Monotonicity: $\mu(A,B)$ is monotone in its arguments. 
\end{enumerate}

Note that $X^*$ in (iii) denotes the conjugate transpose of $X$. It is relatively straightforward to show that the map $\mu(A,B):=A*B$ satisfies properties (i)-(iii) listed above. In the remainder of this section, we will turn our attention to the order and monotonicity properties (iv) and (v).

\subsection{Order and monotonicity properties of $\mu(A,B)=A*B$}

Condition (iv) is a generalization of the property of means of positive numbers whereby a mean of a pair of points is expected to lie between the two points on the number line. For Hermitian matrices, a standard partial order $\preceq $ exists according to which $A\preceq  B$ if and only if $B-A$ is positive semidefinite. This partial order is known as the L{\"o}wner order and the monotonicity of condition (v) is also with reference to this order. It is well-known that the L{\"o}wner order is affine-invariant in the sense that for all $A,B\in \mathbb{P}(n)$, $X\in GL(n)$,
$A\preceq  B $ implies that $XAX^* \preceq  XBX^*$.
In particular, $A\preceq  B$ if and only if $I\preceq  A^{-1/2}BA^{-1/2}$. Thus, by affine-invariance of $\mu$, it suffices to prove (iv) in the case where $A=I$ since 
$
A\preceq \mu(A,B) \preceq B$ if and only if $I \preceq \mu(I,A^{-1/2}BA^{-1/2}) \preceq A^{-1/2}BA^{-1/2}$.
To establish condition (iv) for $\mu(A,B)=A*B$, we shall make use of the following lemma whose proof is elementary.
\begin{lemma} \label{shift}
If $c_1,c_2\in\mathbb{R}$ and $M$ is an $n\times n$ matrix with eigenvalues $\lambda_i(M)$, then $c_1M+c_2I$ has eigenvalues $c_1\lambda_i(M)+c_2$.
\end{lemma}
Let $\Sigma\in\mathbb{P}(n)$ be such that $I\preceq  \Sigma$ and note that this is equivalent to $\lambda_i(\Sigma)\geq 1$ for $i=1,\ldots,n$. Writing $\lambda_{\min}=\lambda_{\min}(\Sigma)$ and $\lambda_{\max}=\lambda_{\max}(\Sigma)$, we have by Lemma \ref{shift} that
\begin{align}
\lambda_i(I*\Sigma)-1&=\lambda_i\left(\frac{1}{\sqrt{\lambda_{\min}}+\sqrt{\lambda_{\max}}}\left(\Sigma+\sqrt{\lambda_{\min}\lambda_{\max}} I\right)\right)-1 \\
&=\frac{\lambda_i(\Sigma)+\sqrt{\lambda_{\min}\lambda_{\max}}-\sqrt{\lambda_{\min}}-\sqrt{\lambda_{\max}}}{\sqrt{\lambda_{\min}}+\sqrt{\lambda_{\max}}} \\
&\geq \frac{\left(\sqrt{\lambda_{\min}}+\sqrt{\lambda_{\max}}\right)\left(\sqrt{\lambda_{\min}}-1\right)}{\sqrt{\lambda_{\min}}+\sqrt{\lambda_{\max}}} \geq 0,
\end{align}
since $\lambda_{i}(\Sigma)\geq \lambda_{\min} \geq 1$. Thus, we have shown that  $I\preceq \Sigma$ implies $I \preceq  I*\Sigma$. To prove the other inequality, note that
\begin{align}
\lambda_i(\Sigma-I*\Sigma)&=\lambda_i\left(\left(\frac{\sqrt{\lambda_{\min}}+\sqrt{\lambda_{\max}}-1}{\sqrt{\lambda_{\min}}+\sqrt{\lambda_{\max}}}\right)\Sigma-\frac{\sqrt{\lambda_{\min}\lambda_{\max}}}{\sqrt{\lambda_{\min}}+\sqrt{\lambda_{\max}}} \; I \right) \\
&= \left(\frac{\sqrt{\lambda_{\min}}+\sqrt{\lambda_{\max}}-1}{\sqrt{\lambda_{\min}}+\sqrt{\lambda_{\max}}}\right)\lambda_i(\Sigma)-\frac{\sqrt{\lambda_{\min}\lambda_{\max}}}{\sqrt{\lambda_{\min}}+\sqrt{\lambda_{\max}}} \\ 
&\geq \frac{(\sqrt{\lambda_{\min}}+\sqrt{\lambda_{\max}}-1)\lambda_{\min}-\sqrt{\lambda_{\min}\lambda_{\max}}}{\sqrt{\lambda_{\min}}+\sqrt{\lambda_{\max}}} \\
& = \sqrt{\lambda_{\min}}\left(\sqrt{\lambda_{\min}}-1\right) \geq 0.
\end{align}
Therefore, we have also shown that $\Sigma-I*\Sigma \succeq  0$. That is, 
\begin{equation} \label{ordering fit}
I\preceq  \Sigma \implies I \preceq  I*\Sigma \preceq  \Sigma,
\end{equation}
for all $\Sigma\in\mathbb{P}(n)$. In particular, upon substituting $\Sigma=A^{-1/2}BA^{-1/2}$ in (\ref{ordering fit}) and using the affine-invariance properties of both the L{\"o}wner order and the mean $\mu(A,B)=A*B$, we establish the following important property.
 
\begin{proposition}
 For $A,B\in\mathbb{P}(n)$, $A\preceq B$ implies that $A\preceq  A*B \preceq  B$.
\end{proposition}

We now consider the monotonicity of $\mu$ in its arguments. First recall that a map $F:\mathbb{P}(n)\rightarrow\mathbb{P}(n)$ is said to be monotone if $\Sigma_1\preceq \Sigma_2$ implies that $F(\Sigma_1)\preceq F(\Sigma_2)$. By symmetry and affine-invariance, it is sufficient to consider monotonicity of $\mu(I,\Sigma)$ with respect to $\Sigma$. That is, monotonicity is established by showing that 
\begin{equation}
\Sigma_1\preceq \Sigma_2 \implies I*\Sigma_1\preceq  \ I*\Sigma_2.
\end{equation}
Unfortunately, it turns out that $F(\Sigma):=I*\Sigma$ is not monotone with respect to $\Sigma$ as we shall demonstrate below. Nonetheless, $F$ is seen to enjoy certain weaker monotonicity properties, which is interesting and insightful. Considering the eigenvalues of $I * \Sigma$, we find that
\begin{align} \label{monotonicity eigenvalues}
\lambda_i(I*\Sigma)=\frac{\lambda_i(\Sigma)+\sqrt{\lambda_{\min}\lambda_{\max}}}{\sqrt{\lambda_{\min}}+\sqrt{\lambda_{\max}}},
\end{align}
where $\lambda_{\min}$ and $\lambda_{\max}$ refer to the smallest and largest eigenvalues of $\Sigma$. 
\begin{proposition} \label{min max monotone}
The maximum and minimum eigenvalues of $F(\Sigma)=I*\Sigma$ are monotone with respect to $\Sigma$.
\end{proposition}
\begin{proof}
Considering the cases $i=1$ and $i=n$, we find that
(\ref{monotonicity eigenvalues}) yields
\begin{equation}
\lambda_{\min}(I*\Sigma) = \sqrt{\lambda_{\min}(\Sigma)} \quad \mathrm{and} \quad \lambda_{\max}(I*\Sigma) = \sqrt{\lambda_{\max}(\Sigma)},
\end{equation}
both of which are seen to be monotone functions of $\Sigma$. 
\qed
\end{proof}
It is in the sense of the above that $\mu(A,B)=A*B$ inherits a weak monotonicity property. The monotonic dependence of the extremal eigenvalues of $I*\Sigma$ on $\Sigma$ ensures that if $\Sigma_1\preceq \Sigma_2$, then we can at least rule out the possibility that $I*\Sigma_1 \succ I*\Sigma_2$, where $\succ 0$ here denotes positive definiteness. To prove that monotonicity is generally not satisfied in the full sense of condition (v), consider a diagonal matrix $\Sigma=\operatorname{diag}(a,b,x)\in\mathbb{P}(3)$, where $\lambda_{\min}(\Sigma)=a < b \leq x = \lambda_{\max}(\Sigma)$ and $x$ is thought of as a variable. We have $I*\Sigma = \operatorname{diag}\left(\sqrt{a},f(x),\sqrt{x}\right)$,
where 
\begin{equation}
\lambda_2(I*\Sigma)=f(x):=\frac{b+\sqrt{ax}}{\sqrt{a}+\sqrt{x}}.
\end{equation}
Taking the derivative of $f$ with respect to $x$, we find that 
\begin{equation}
f'(x)=\frac{a-b}{2\sqrt{x}(\sqrt{a}+\sqrt{x})^2} < 0, \quad \forall x \geq b,
\end{equation}
which shows that the second eigenvalue of $I*\Sigma$ decreases as $x$ increases. Thus, we see that $I*\Sigma$ cannot depend monotonically on $\Sigma$ in this example. 

\subsection{A geometric scaling property}

Before completing this section on the midrange $\mu(A,B)=A*B$ of a pair of positive definite matrices, we note a key scaling property satisfied by $\mu$ which suggests that it is a plausible candidate as a scalable substitute for the standard matrix geometric mean $A\#B$.

\begin{proposition}
For any real scalars $a,b>0$ and matrices $A,B\in\mathbb{P}(n)$, we have 
\begin{equation} \label{geo scaling prop}
(aA)*(bB)=\sqrt{ab}(A*B).
\end{equation}
\end{proposition}
\begin{proof}
The result follows upon substituting $\lambda_i\left((bB)(aA)^{-1}\right)=\frac{b}{a}\lambda_i(BA^{-1})$ into the formula (\ref{star}).
\qed
\end{proof}

The scaling in (\ref{geo scaling prop}) of course does not generally hold for a mean of two matrices. Indeed, it does not generally hold for means arising as $d_{\infty}$-midpoints either. For instance, \cite{Lim2013} identifies 
\begin{equation} \label{diamond}
A\diamond B = 
\begin{dcases}
\frac{\sqrt{\lambda_{\max}}}{1+\lambda_{\max}}(A+B) \quad \mathrm{if} \quad \lambda_{\min}\lambda_{\max} \geq 1, \\
\frac{\sqrt{\lambda_{\min}}}{1+\lambda_{\min}}(A+B) \quad \mathrm{if} \quad \lambda_{\min}\lambda_{\max} \leq 1,
\end{dcases}
\end{equation}
as another $d_{\infty}$-midpoint of $A$ and $B$ corresponding to the midpoint of a different $d_{\infty}$ minimal geodesic to the one considered in this paper. It is clear that (\ref{diamond}) does not scale geometrically as in (\ref{geo scaling prop}).
As a summary, we collect the key results of this section in the following theorem.

\begin{theorem}
The mean $\mu(A,B)=A*B$ defined in (\ref{star}) yields a $d_{\infty}$-midpoint of $A,B\in\mathbb{P}(n)$ that is continuous, symmetric, affine-invariant, and scales geometrically as in (\ref{geo scaling prop}). 
Moreover, if $A\preceq  B$, then $A\preceq \mu(A,B) \preceq  B$, and the extremal eigenvalues of $\mu(I,\Sigma)$ depend monotonically on $\Sigma\in\mathbb{P}(n)$.
\end{theorem}

\section{The midrange problem in $\mathbb{R}_+$ and $\mathbb{R}_+^n$}
\label{sec: 3}

It is instructive to consider the $N$-point affine-invariant midrange problem in the scalar and vector cases, corresponding to the cones $\mathbb{R}_+$ and $\mathbb{R}_+^n$, respectively. In the scalar case, we are given $N$ positive numbers $y_i>0$ that can be ordered such that $\min_i y_i \leq y_k \leq \max_i y_i$ for each $k=1, \ldots, N$. By the monotonicity of the $\log$ function, we have
$\log\left(\min_i y_i\right) \leq \log y_k \leq \log\left(\max_i y_i\right)$.
The midrange is uniquely given by
\begin{equation} \label{scalar}
x=\exp\left(\frac{1}{2}\left[\log\left(\min_i y_i\right)+\log\left(\max_i y_i\right)\right]\right) = \left(\min_i y_i \cdot \max_i y_i\right)^{1/2}.
\end{equation}
Note that (\ref{scalar}) is the unique solution of the optimization problem
\begin{equation}
\min_{x>0} \; \max_i \; |\log x - \log y_i|.
\end{equation}
In the vector case, the midrange problem in $\mathbb{R}^n_+$ takes the form
\begin{equation} \label{vector midrange problem}
\min_{\boldsymbol{x}>0} \; \max_i \; \|\log\boldsymbol{x}-\log\boldsymbol{y}_i\|_{\infty} := \min_{\boldsymbol{x}>0} \; \max_i \; \max_a \;|\log x^a-\log y^a_i| ,
\end{equation}
where $\boldsymbol{x}>0$ means that $\boldsymbol{x}=(x^a)$ satsifies $x^a>0$ for $a=1,\ldots,n$ and $\boldsymbol{y_i}$ are a collection of $N$ given points in $\mathbb{R}^n_+$. As in the matrix case, the optimum cost $t^{\star}$ has a lower bound 
\begin{equation} \label{lower vector}
l=\frac{1}{2}\max_{i,j}\|\log\boldsymbol{y}_i-\log\boldsymbol{y}_j\|_{\infty}.
\end{equation}

\begin{proposition}
The lower bound (\ref{lower vector}) is attained by $\boldsymbol{x}=(x^a)\in\mathbb{R}^n_+$ given by
\begin{equation} \label{vector midpoint}
x^a=\left(\min_i y_i^a \cdot \max_i y_i^a\right)^{1/2}.
\end{equation}
\end{proposition}

\section{The $N$-point midrange problem in $\mathbb{P}(n)$}
\label{sec: 4}

In the matrix setting, the midrange problem (\ref{midrange}) takes the form 
\begin{align}  \label{matrix midrange}
\min_{X \succeq 0} \; \max_i \; \|\log(Y_i^{-1/2}XY_i^{-1/2})\|_{\infty},
\end{align}
which can be rewritten as
\begin{equation}  \label{matrix midrange 2}
\begin{cases}
\min_{X \succeq 0,t} \; t \\
e^{-t}Y_i \preceq X \preceq e^{t}Y_i
\end{cases}
\end{equation}
While this problem is not convex due to the presence of the $\log$ function, the feasibility condition $e^{-t}Y_i \preceq X \preceq e^{t}Y_i$ is convex for \emph{fixed} $t$ and can be solved using standard convex optimization packages such as CVX \cite{cvx}. Given a $t$ that is greater than or equal to the optimum value $t^{\star}=\min_X  \max_i \; d_{\infty}(X,Y_i)$, we can solve (\ref{matrix midrange 2}) by successively solving the feasibility condition as we decrease $t$. In the bisection method it is very desirable to have a good estimate for the initial $t$ as the successive reductions in $t$ can be quite slow. In particular, if the lower bound $l=\frac{1}{2}\operatorname{diam}(\{Y_i\})$ is attained as in the vector case, then we can solve (\ref{matrix midrange 2}) in one step by taking $t=l$ and solving the feasibility condition once. Unfortunately, and rather remarkably, numerical examples show that unlike the scalar and vector cases, the lower bound $l$ is not always attained in the affine-invariant matrix midrange problem. 

\begin{proposition}
The lower bound $l=\frac{1}{2}\operatorname{diam}_{\infty}(\{Y_i\})$ is not necessarily attained in (\ref{matrix midrange}).
\end{proposition}

The above result suggests that the $N$-point matrix midrange problem is more challenging than the vector case in fundamental ways. While the bisection method offers a solution to this problem, we expect that significantly more efficient solutions to the problem can be found. In particular, we expect to find algorithms that rely principally on the computation of dominant generalized eigenpairs that would be considerably more efficient and scalable than existing algorithms for computing matrix geometric means, as in the $N=2$ case.

%
%
%
\bibliographystyle{splncs04}
 \bibliography{references}

\end{document}